\documentclass[11pt]{amsart} 
\usepackage{amssymb,amsmath,accents}
\usepackage{amscd}
\usepackage{amsfonts,amsthm,mathrsfs}
\usepackage{setspace}

\usepackage{graphicx, xcolor}

\newtheorem{thm}{Theorem}
\newtheorem{cor}[thm]{Corollary}
\newtheorem{lem}[thm]{Lemma}
\newtheorem{prop}[thm]{Proposition}

\newtheorem{rem}[thm]{Remark}

\makeatletter
\@namedef{subjclassname@2020}{\textup{2020} Mathematics Subject Classification}
\makeatother

\begin{document}

\title[Forward self-similar solutions to the surface diffusion flow]
{Remarks on graph-like forward self-similar solutions \\ to the surface diffusion flow equations}

\author[Y.~Giga]{Yoshikazu Giga}
\address[Y.~Giga]{Graduate School of Mathematical Sciences, The University of Tokyo, 3-8-1 Komaba, Meguro-ku, Tokyo 153-8914, Japan.}
\email{labgiga@ms.u-tokyo.ac.jp}

\author[S.~Katayama]{Sho Katayama}
\address[S.~Katayama]{Graduate School of Mathematical Sciences, The University of Tokyo, 3-8-1 Komaba, Meguro-ku, Tokyo 153-8914, Japan.}
\email{katayama-sho572@g.ecc.u-tokyo.ac.jp}




\begin{abstract}
We clarify existence and non-existence of graph-like forward self-similar solutions to the planar surface diffusion equations.
\end{abstract}

\maketitle
\thispagestyle{empty}

\section{Introduction} \label{SIN}

We consider the surface diffusion equation of the form
\begin{equation} \label{ESD}
	V = - \Delta_{\Gamma_t} H
	\quad\text{on}\quad \Gamma_t
\end{equation}
for an evolving hypersurface $\Gamma_t$ in $\mathbb{R}^d$; 
 here, $V$ denotes the normal velocity of $\Gamma_t$ in the direction of a unit normal vector field $\mathbf{n}$ of $\Gamma_t$ and $H$ denotes the $(d-1)$ times mean curvature of $\Gamma_t$ in the direction of $\mathbf{n}$.
 The operator $\Delta_{\Gamma_t}$ denotes the Laplace--Beltrami operator of $\Gamma_t$, that is,
\[
	\Delta_{\Gamma_t} = \operatorname{div}_{\Gamma_t} \nabla_{\Gamma_t},
\]
where $\operatorname{div}_{\Gamma_t}$ denotes the surface divergence and $\nabla_{\Gamma_t}$ denotes the surface gradient.
 By this notation, $H=-\operatorname{div}_{\Gamma_t}\mathbf{n}$.

We say that a solution $\Gamma_t$ of \eqref{ESD} is (forwardly) \emph{self-similar} if $\Gamma_t$ is of the form
\[
	\Gamma_t = t^{1/4} \Gamma_*, \quad
	t > 0
\]
with some hypersurface $\Gamma_*$ independent of time.
 We say that $\Gamma_*$ is a \emph{profile} surface.
 A direct calculation shows that $\Gamma_*$ satisfies
\begin{equation} \label{EFS}
	\frac{x\cdot\mathbf{n}}{4} = - \Delta_{\Gamma_*} H
	\quad\text{on}\quad \Gamma_*.
\end{equation}
We are interested in existence of a non-trivial self-similar solution.
 We say that $\Gamma_*$ is graph-like if the profile surface $\Gamma_*$ is of the form with some function $\phi$ on $\mathbb{R}^{d-1}$, that is,
\[
	\Gamma_* = \left\{ \left(x', \phi(x')\right) \in \mathbb{R}^d \Bigm|
	x' \in \mathbb{R}^{d-1} \right\}.
\]
In \cite{KL}, among other results, Koch and Lamm proved that for a given Lipschitz function $\Omega(\omega)$ on 
a unit sphere $S^{d-2}$ in $\mathbb{R}^{d-1}$, there always exists a unique (smooth) graph-like self-similar solution of the form 
\begin{equation} \label{ES1}
	\lim_{r\to\infty} \sup_{\omega\in S^{d-2}} \left| \phi(r\omega)/r - \Omega (\omega) \right| = 0
\end{equation}
provided that 
\[
	\sup_{\omega\in S^{d-2}} \left(\left| \Omega (\omega) \right|
+ \left| \nabla_{S^{d-2}} \Omega (\omega) \right|\right)
\]
is sufficiently small.
 Actually $\phi$ is an analytic function.
 They constructed such a solution by solving an evolution equation \eqref{ESD} for a graph starting from a homogeneous data $\phi_0(x)=\Omega \left(x/|x|\right)|x|$.
 The idea to construct a self-similar solution to an evolution equation by solving the initial value problem with homogeneous initial data goes back to \cite{GM}, where non-trivial self-similar solutions are constructed to the vorticity equations which are formally equivalent to the Navier-Stokes equations.
 Du and Yip \cite{DY} proved that such a small self-similar solution constructed in \cite{KL} is stable under small perturbation of initial data by adapting a compactness argument which goes back to \cite{GGS}, where large time behavior of solutions to two-dimensional vorticity equations is discussed.
 Recently, it was proved by \cite{GGK} that under a suitable smallness condition for the initial data, a solution of exponential type surface diffusion equation like $V=\Delta_\Gamma\exp(-H)$ behaves asymptotically close to the self-similar solution of \cite{KL} as time tends to infinity when $d=2$;
 their proof works for higher dimensions as remarked in \cite{GGK}. 

In a recent paper, P.\ Rybka and G.\ Wheeler \cite{RW} established interesting non-existence results of several special graph-like solutions when $d=2$.
They proved that there is no stationary graph-like solution to \eqref{ESD} other than line.
 Moreover, they showed that all graph-like traveling wave type solution like a soliton must be linear.
 For a forward graph-like self-similar solution, under additional conditions, they proved that the profile function $\phi$ must be linear.
 In \cite[Remark 16]{RW}, they gave an impression that 
there exist no graph-like self-similar solutions other than linear one when $d=2$.
 The purpose of this paper is to clarify this apparently contradicting status.

 This paper is organized as follows.
 In Section \ref{SKI}, we derive a key identity for forward self-similar solutions essentially derived by \cite{RW}.
 We here derive its geometric form in general dimensions.
 In Section \ref{SEC}, we give a 
sufficient condition so that a profile function is linear.
 In Section \ref{SFP}, we recall main results and proofs of \cite{RW} for a graph-like forward self-similar solution.
 We also point out that their interpretation \cite[Remark 16]{RW} of their results is misleading.
 In Section \ref{SCE}, we recall a main result of \cite{KL} concerning the existence of non-trivial (nonlinear) graph-like small forward self-similar solutions.
 As pointed out by \cite[Remark 16]{RW}, the statement would give an impression that their solution (Theorem \ref{TKL}) is self-similar \emph{up to} non-zero time-dependent spatially constant function.
 We clarify this point and it turns out that there is no such ambiguity.

\section{A key identity for a profile surface} \label{SKI} 

We begin with several elementary identities. 
\begin{prop} \label{PEI}
Let $\Gamma$ be a given hypersurface in $\mathbb{R}^d$.
 Let $a$ be a $C^1$ function on $\Gamma$.
 Then
\begin{enumerate}
\item[(i)] $\operatorname{div}_\Gamma x = d-1$,
\item[(i\hspace{-0.1em}i)] $\operatorname{div}_\Gamma \left( a(x)\mathbf{n} \right) = -a(x) H$,
\item[(i\hspace{-0.1em}i\hspace{-0.1em}i)] $\nabla_\Gamma \left( |x|^2/2 \right) = x - \mathbf{n}(x \cdot \mathbf{n})$,
\item[(i\hspace{-0.1em}v)] $\Delta_\Gamma \left( |x|^2/2 \right) = d - 1 + H(x \cdot \mathbf{n})$.
\end{enumerate}
\end{prop}
\begin{proof}
We write $\mathbf{n}=(n_1,\ldots,n_d)$, $x=(x_1,\ldots,x_d)$.
\begin{enumerate}
\item[(i)] By definition (see e.g.\ \cite{G}),
\begin{align*}
	\operatorname{div}_\Gamma x &= \operatorname{trace} (I-\mathbf{n}\otimes\mathbf{n}) Dx
	= \sum_{1\le i,j\le d} (\partial_i - n_i n_j \partial_j) x_i \\
	&= \operatorname{div}x - \sum_{1\le i,j\le d} n_i n_j \delta_{ij}
	= d - |\mathbf{n}|^2 = d-1,
\end{align*}
where $\partial_i=\partial/\partial x_i$.
\item[(i\hspace{-0.1em}i)] By Leibniz's rule,
\[
	\operatorname{div}_\Gamma(a\mathbf{n})
	= a \operatorname{div}_\Gamma \mathbf{n}
	+ (\nabla_\Gamma a) \cdot \mathbf{n}.
\]
Since $\nabla_\Gamma a$ is a tangent vector, the second term vanishes so (i\hspace{-0.1em}i) follows.
\item[(i\hspace{-0.1em}i\hspace{-0.1em}i)] By definition,
\begin{align*}
	\nabla_\Gamma \left(|x|^2/2\right)
	&= \left( \left(\partial_i - \sum_{j=1}^d n_i n_j \partial_j \right) \left(|x|^2/2\right) \right)_{i=1}^d \\
	& = \left( x_i - \sum_{j=1}^d n_i n_j x_j \right)_{i=1}^d = x-\mathbf{n}(x\cdot\mathbf{n}).
\end{align*}
\item[(i\hspace{-0.1em}v)] This follows from (i), (i\hspace{-0.1em}i), (i\hspace{-0.1em}i\hspace{-0.1em}i).
 Indeed,
\begin{align*}
	\Delta_\Gamma \left( |x|^2/2 \right) 
	&= \operatorname{div}_\Gamma \left( \nabla_\Gamma \left(|x|^2/2 \right)\right) \\
	&= \operatorname{div}_\Gamma \left( x - \mathbf{n}(x\cdot\mathbf{n}) \right)  \quad\text{by (i\hspace{-0.1em}i\hspace{-0.1em}i)}\\
	&= d-1 + H(x\cdot\mathbf{n}) \quad\text{by (i) and (i\hspace{-0.1em}i)}.
\end{align*}
\end{enumerate}
\end{proof}

 We next derive a key identity for self-similar solutions.
 As remarked in Remark \ref{RRW}, the same identity is derived by \cite{RW} where $\Gamma_*$ is a graph-like curve.  
\begin{lem} \label{LSI}
If $\Gamma_*$ is a profile surface of a (forward) self-similar solution to \eqref{ESD}, then
\[
	\Delta_{\Gamma_*} \left( H^2 + \frac{|x|^2}{4} \right)
	= \frac{d-1}{2} + 2|\nabla_{\Gamma_*} H|^2 \quad\text{on}\quad \Gamma_*.
\]
\end{lem}
\begin{proof} 
Since
\[
	\Delta_{\Gamma_*} H^2 = 2 \operatorname{div}_{\Gamma_*} (H \nabla_{\Gamma_*} H)
	= 2H \Delta_{\Gamma_*} H + 2|\nabla_{\Gamma_*} H|^2,
\]
it follows from Proposition \ref{PEI} that
\[
	\Delta_{\Gamma_*} \left( H^2 + \frac{|x|^2}{4} \right)
	= \frac{d-1}{2} + 2|\nabla_{\Gamma_*} H|^2 + H \left( \frac{x \cdot \mathbf{n}}{2} + 2\Delta_{\Gamma_*} H \right).
\]
The desired identity follows since $\Gamma_*$ satisfies \eqref{EFS}.
\end{proof}
\begin{rem} \label{RB}
If we consider a backward self-similar solution, $\Gamma_*$ must satisfy
\[
	\frac{x \cdot \mathbf{n}}{4} = \Delta_{\Gamma_*} H_* \quad\text{on}\quad \Gamma_*.
\]
In this case,
\begin{align*}
	\Delta_{\Gamma_*} \left( H^2 - \frac{|x|^2}{4} \right) 
	&= - \frac{d-1}{2} + 2 |\nabla_{\Gamma_*} H|^2 + H \left( -\frac{x\cdot\mathbf{n}}{2} + 2 \Delta_{\Gamma_*} H \right) \\
	&= - \frac{d-1}{2} + 2 |\nabla_{\Gamma_*} H|^2.
\end{align*}
\end{rem}
\begin{rem} \label{RRW}
In the case $d=2$, $\nabla_\Gamma$ and $\operatorname{div}_\Gamma$ are just differentiations with respect to the arc-length parameter which is denoted by $\partial_s$.
 The identity in Lemma \ref{LSI} is now
\begin{equation} \label{E2D}
	\partial_s^2 \left( k^2 {+} \frac{|x|^2}{4} \right) 
	= \frac12 + 2(\partial_s k)^2,
\end{equation}
where $k$ denotes the curvature.
 If $\Gamma_*$ is given as a graph, i.e., $x_2=\phi(x_1)$, the upward curvature $k$, $\partial_s$ and $|x|^2$ are of the form
\[
	k[\phi] = \frac{\partial_{x_1}^2 \phi}{v^3}, \quad 
	v[\phi] = \sqrt{1 + (\partial_{x_1} \phi)^2}, \quad
	\partial_s = \frac1v \partial_{x_1}, \quad
	|x|^2 = x_1^2 + \phi (x_1)^2.
\]
By using this notation, \cite{RW} derived \eqref{E2D}.
\end{rem}

As an application of Lemma \ref{LSI}, we prove the non-existence of a compact profile function.
\begin{cor} \label{CNE}
There is no smooth compact profile surface of forward self-similar solution to \eqref{ESD}.
\end{cor}
\begin{proof}
If $\Gamma_*$ is compact, $H^2+|x|^2/4$ takes its maximum on $\Gamma_*$.
 At a maximum point, 
\[
	\Delta_{\Gamma_*}\left(H^2+|x|^2/4\right) \le 0.
\]
This contradicts the identity of Lemma \ref{LSI}.
\end{proof}

\section{Estimate of $|x|^2$ and its consequence} \label{SEC} 

{From now on we concentrate on graph-like solutions for $d=2$.}
 We set the arc-length parameter
\[
	s(x_1) = \int_0^{x_1} \sqrt{1 + (\partial_{x_1}\phi)^2 (z)}\,dz
	= \int_0^{x_1} v [\phi] (z)\,dz.
\]
Note that $s<0$ if $x_1<0$.
 We begin with an estimate from elementary geometry.
 See Figure \ref{Fx}.
\begin{figure}[htb]
\centering
\includegraphics[width=6.5cm]{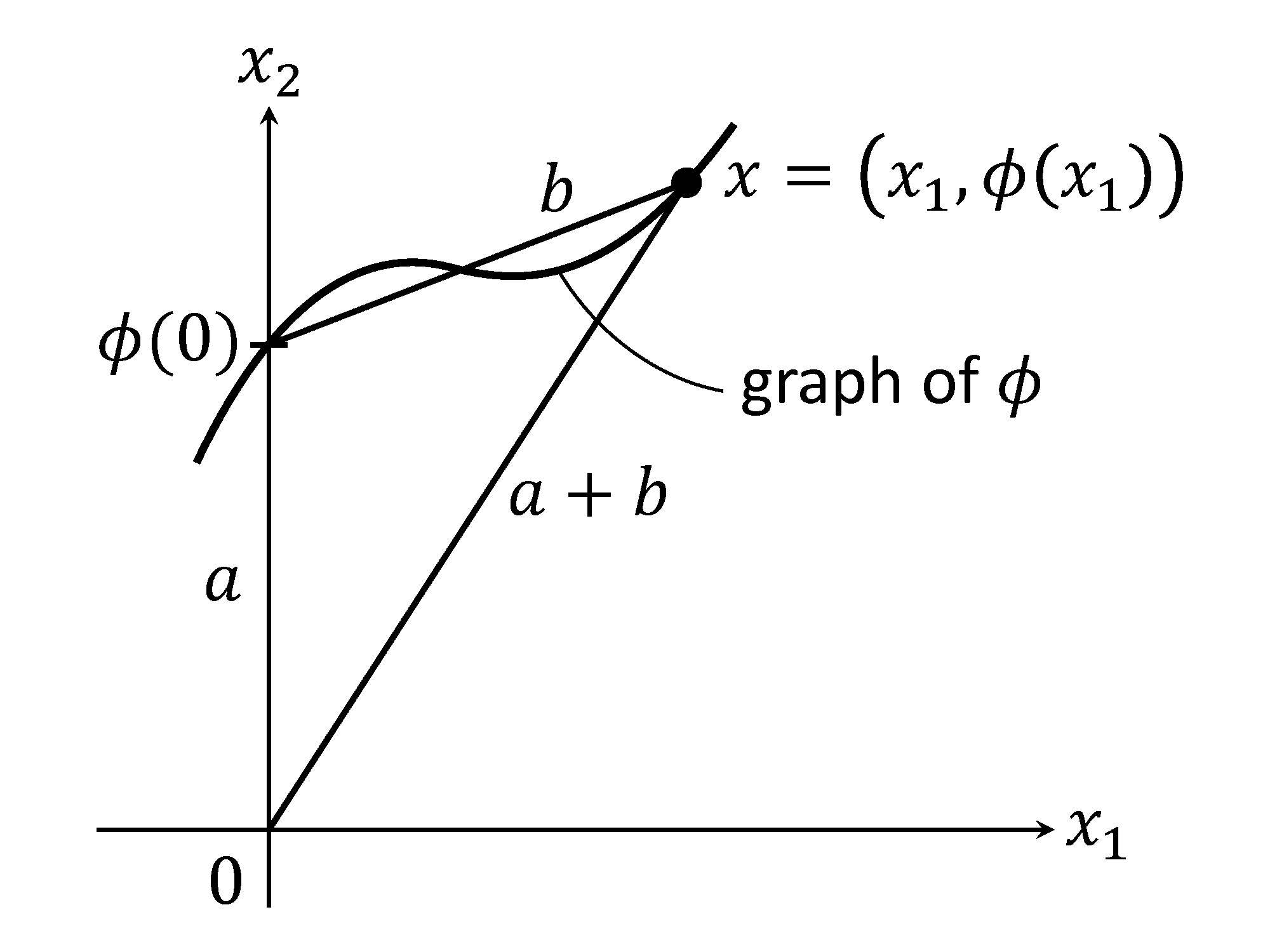}
\caption{location of $x$} \label{Fx}
\end{figure}
\begin{prop} \label{PEG}
For $x=\left(x_1,\phi(x_1)\right)$, $|x|^2$ is estimated as 
\[
	|x|^2 \le s^2 + \phi(0)^2 + 2\phi(0) \left(\phi(x_1)-\phi(0)\right)
	= s^2 - \phi(0)^2 + 2\phi(0) \phi(x_1), \quad
	{x_1 \in \mathbb{R}.}
\]
\end{prop}
\begin{proof}
We set $a=\left(0,\phi(0)\right)$, $b=\left(x_1,\phi(x_1)-\phi(0)\right)$ so that $a+b=x=\left(x_1,\phi(x_1)\right)$.
 Then
\begin{align*}
	|x|^2 = |a+b|^2 &= |a|^2 + |b|^2 + 2a \cdot b \\
	&\le \phi(0)^2 + s^2 + 2\phi(0) \left(\phi(x_1)-\phi(0)\right)
\end{align*}
since $|b|\le|s|$.
\end{proof}

Combining this with the identity \eqref{E2D} which is a two-dimensional graph version of that in Lemma \ref{LSI}, we conclude that the profile function is linear provided that $\phi$ satisfies a special growth estimate \ref{ESp} {below.}
\begin{thm} \label{TNon}
Let $\phi$ be a profile function of a graph-like (forward) self-similar solution of \eqref{ESD}.
 If $\phi$ satisfies
\begin{equation} \label{ESp}
	\phi(0) \phi(x_1) \le \alpha s(x_1) + \beta, \quad
	x_1 \in \mathbb{R}
\end{equation}
for some $\alpha,\beta\in\mathbb{R}$, then $\phi(x_1)=cx_1$ for some $c\in\mathbb{R}$.
\end{thm}
\begin{proof}
We basically follow the idea of \cite{RW}.
 By Proposition~\ref{PEG} and \eqref{ESp}, we observe that
\[
	|x|^2 - \left( s^2 - \phi(0)^2 + 2\alpha s + 2\beta \right) \le 0.
\]
We set
\[
	{Q_{\alpha\beta}}(x_1) := k^2 + \frac14 \left\{ |x|^2 - s^2 + \phi(0)^2 - 2\alpha s - 2\beta \right\}.
\]
Then we see that
\begin{equation} \label{EKB}
	k^2 \ge {Q_{\alpha\beta}}(x_1).
\end{equation}
By \eqref{E2D}, we observe that
\begin{equation} \label{ECon}
	\partial_s^2 {Q_{\alpha\beta}}(x_1) = 2(\partial_sk)^2 \ge 0.
\end{equation}
As a function of $s$, ${Q_{\alpha\beta}}$ is convex.
 Unless ${Q_{\alpha\beta}}$ is a constant, ${Q_{\alpha\beta}}\to\infty$ either $s\to+\infty$ or $-\infty$.
 By \eqref{EKB}, $|k|$ is bounded away from zero for sufficiently large $x_1$ or $-x_1$ if ${Q_{\alpha\beta}}$ is not a constant.
 Since $\Gamma_*$ is a graph of $\phi$, this is impossible (see e.g.\ \cite[Lemma 6]{RW}).
 Thus ${Q_{\alpha\beta}}$ must be a constant.
 By \eqref{ECon}, this implies that $k\equiv0$, that is,
\[
	\phi(x_1) = cx_1 + c'.
\]
Since $x\cdot\mathbf{n}=\left(-(\partial_{x_1}\phi)x_1+\phi\right)/v$ and $\partial_s^2k=0$, \eqref{EFS} implies that
\[
	x\cdot\mathbf{n} = 0 \quad\text{i.e.,}\quad
	-(\partial_{x_1}\phi) x_1 + \phi = 0.
\]
In particular, $\phi(0)=0$.
 Thus $\phi(x_1)=cx_1$,
\end{proof}
\begin{rem} \label{RRWE}
 Suppose that $\phi(0)\left(\phi(x_1)-\phi(0)\right)$ is estimated as
\begin{equation} \label{Ekey}
	2\phi(0) \left(\phi(x_1)-\phi(0)\right)
	\le \left|\phi(0)\right| s(x_1) 
	 \quad\text{for}\quad x_1 \in \mathbb{R}.
\end{equation}
{We} set
\[
	Q(x_1) := k^2 + \frac14 \left\{ |x|^2 - s^2 - \phi(0)^2 - 2\left|\phi(0)\right| s \right\}.
\]
By Proposition \ref{PEG}, we obtain \eqref{EKB} and \eqref{ECon}.
 It is claimed in an old version of \cite{RW} that the estimate \eqref{Ekey} holds for all $C^1$ function $\phi$ on $\mathbb{R}$. As in the proof of Theorem~\ref{TNon}, they concluded that all profile functions of graph-like forward self-similar solutions are of the form $\phi(x_1)=cx_1$.
 Unfortunately, \eqref{Ekey} does not hold for $x_1<0$.
 The correct estimate is
\[
	2\phi(0) \left( \phi(x_1) - \phi(0) \right)
	\le \left| \phi(0)\right| \left|s(x_1)\right|.
\]
With this change, $Q$ should be altered as
\[
	Q(x_1) = k^2 + \frac14 \left\{ |x|^2 - s^2 - \phi(0)^2 - 2\left|\phi(0)\right| |s| \right\}
\]
to get \eqref{EKB}.
 However, \eqref{ECon} is fulfilled only for $x_1>0$, $x_1<0$.
 Actually, $\partial_s^2Q(x_1)$ has a negative delta part at $x_1=0$.
 This breaks the whole argument. 
\end{rem}

It is interesting to discuss when \eqref{ESp} is actually fulfilled.
 We give a sufficient condition.
\begin{lem} \label{LSSp}
Let $\phi\in L_{loc}^1(\mathbb{R})$ satisfy
\[
	\phi' - a_0 \in L^1(\mathbb{R})
\]
for some $a_0\in\mathbb{R}$.
 Then
\[
	\phi(z) \le a_0 z + \beta, \quad
	z \in \mathbb{R},
\]
with a constant $\beta=\int_{-\infty}^\infty|\phi'-a_0|\,d\tau$.
\end{lem}
\begin{proof}
Since $\phi'-a_0\in L^1(\mathbb{R})$, we see that
\[
	\phi(z) - a_0 z = - \int_z^\infty(\phi'-a_0)\,d\tau.
\]
Estimating the right-hand side by $\beta=\int_{-\infty}^\infty|\phi'-a_0|\,d\tau$, we conclude the desired estimate.
\end{proof}
\begin{cor} \label{CNon}
Assume that $\phi\in C^1(\mathbb{R})$ be a profile function of a self-similar solution to \eqref{ESD}.
 Assume that $\phi$ is globally Lipschitz and
\[
	\partial_{x_1}\phi - a_0 \in L^1(\mathbb{R})
\]
with some $a_0\in\mathbb{R}$.
 Then $\phi(x_1)=a_0 x_1$.
\end{cor}
%
%

\section{{Further properties of profile functions}} \label{SFP} 

P.\ Rybka and G.\ Wheeler \cite{RW} derived a few properties of a profile of a graph-like self-similar solution as an interesting application of Lemma \ref{LSI} or Remark \ref{RRW}.
 Let us recall their results with slightly different terminology.

\begin{thm} \label{TPRW}
\begin{enumerate}
\item[(i)] Let $\Gamma_*$ be a profile curve of a graph-like forward self-similar solution.
 If $\Gamma_*$ contains the origin, then $\Gamma_*$ must be a line.
 In other words, if $\Gamma_*$ is written as the graph of $\phi$ and $\phi(0)=0$, then $\phi$ is linear, i.e., $\phi(x_1)=cx_1$ with some $c\in\mathbb{R}$.
\item[(i\hspace{-0.1em}i)] Let $\phi$ be a profile function of graph-like forward self-similar solution.
 Let us define
\[
	D_0[\phi](x_1) := \phi(x_1)^2 + x_1^2
	- \left( \int_0^{x_1} \sqrt{1+\phi'(z)^2} \,dz
	+ \left| \phi(0) \right| \right)^2.
\]
Then, either
\[
	\varlimsup_{x_1\to+\infty} D_0[\phi](x_1) = \infty
	\quad\text{or}\quad
	\varlimsup_{x_1\to-\infty} D_0[\phi](x_1) = \infty
\]
unless $\phi$ is linear.
\end{enumerate}
\end{thm}

Let us recall the proof in \cite{RW}.
 We set
\[
	Q(x_1) = k^2 + \frac{|x|^2}{4}
	- \frac14 \left( s + \left| \phi(0) \right| \right)^2
\]
with $x=\left(x_1,\phi(x_1)\right)$, $s=\int_0^{x_1}\sqrt{1+\phi'(z)^2}\,dz$, $k=\partial_x^2\phi/\left(1+\phi'(z)^2\right)^{3/2}$.
 If $\phi$ is a profile function of a (graph-like) forward self-similar solution, then by Lemma \ref{LSI} or Remark \ref{RRW}, $Q$ must be convex with respect to the variable $s$.
 If $\phi(0)=0$, then by geometry (Proposition \ref{PEG}) we see that $|x|^2\le s^2$.
 In other words, $D_0[\phi]\le0$.
 Thus
\[
	Q(x_1) = k^2 + \frac14 D_0[\phi] \le k^2.
\]
Since $Q$ is convex with respect to $s$, either
\[
	\varlimsup_{x_1\to+\infty} Q(x_1) = \infty
	\quad\text{or}\quad
	\varlimsup_{x_1\to-\infty} Q(x_1) = \infty
\]
unless $Q$ is a constant.
 This implies that $k(x_1)\to\infty$ as either $x_1\to\infty$ or $x_1\to-\infty$ unless $\phi$ is linear.
 However, such behaviors of $k$ are impossible for a graph-like function (\cite[Lemma 7]{RW}).
 Thus $\phi$ must be linear, which proves (i).

The proof for (i\hspace{-0.1em}i) is similar for a forward self-similar solution.
 In this case, if $D_0[\phi]$ is bounded from above, then
\[
	Q(x_1) \le k^2 + \text{a bounded function.}
\]
Then, as in the proof of (i), unless $Q$ is a constant, $k(x_1)\to\infty$ as either $x_1\to\infty$ or $x_1\to-\infty$, which cannot happen.
\begin{rem} \label{RBS}
	In the case of backward self-similar solution, instead of (i\hspace{-0.1em}i) we are able to assert that
\[
	\varliminf_{x_1\to\infty} D_0[\phi](x_1) = -\infty
	\quad\text{or}\quad
	\varliminf_{x_1\to-\infty} D_0[\phi](x_1) = -\infty
\]
unless $\phi$ is linear.
 This can be proved by setting
\[
	\tilde{Q}(x_1) = k^2 - \frac14 D_0[\phi]
\]
and observing $\tilde{Q}$ is convex with respect to $s$ by Remark \ref{RB}.
 In \cite{RW}, instead of one side boundedness, they assume that $D_0$ is bounded.
\end{rem}
\begin{rem} \label{RCRW}
Since the term of degree one in $s$ in $D_0$ is not a leading term as $s\to\infty$, the boundedness of $D_0$ is equivalent to saying that
\begin{equation} \label{EBD}
	D[\phi](x_1) := \phi(x_1)^2 + x_1^2
	- \left( \int_0^{x_1} \sqrt{1+\phi'(z)^2}\, dz\right)^2
	\ \text{is bounded.}
\end{equation}
In \cite[Remark 16]{RW}, it is alluded that there is no graph-like forward self-similar solution close to a homogeneous function except linear one, which apparently would contradict the existence result of \cite{KL}.
 We remark here that the boundedness condition \eqref{EBD} of $D$ is quite unstable and it seems that there is no reasonable neighborhood near a homogeneous function so that $D$ is bounded.

Here is our explanation.
 In \cite[Remark 16]{RW}, the authors explained that since a homogeneous function
\begin{equation}\label{homo}
  \phi_{A,B}(y):=\begin{cases}
    Ay,& y\ge 0,\\
    -By,& y<0
  \end{cases}
\end{equation}
with $A,B\in\mathbb{R}$, satisfies the condition {\eqref{EBD}}, smooth functions $\phi$ which are ``close'' to $\phi_{A,B}$ also satisfy the condition {\eqref{EBD}}. However, this observation {is misleading.}
 $D[\phi](y)$ is expressed by
\[
D[\phi](y)=(|x|+s_0)(|x|-s_0)
\]
where $x:=(y,\phi(y))$ and $\displaystyle s_0:=\int_0^y\sqrt{1+\phi'(z)^2}\,dz$. Here, for $y>0$, $|x|$ and $s_0$ are both greater than $|y|$. This implies that for $D[\phi](y)$ to be bounded for $y\in\mathbb{R}$, $\phi$ has to satisfy
\begin{equation}\label{mustasymp}
||x|-s_0|\le \frac{C}{|y|}\ \textrm{for $y>0$ large.}
\end{equation}
But this condition is not satisfied in general, even if $\phi$ is sufficiently close to some homogeneous function $\phi_{A,B}$. We explain this by an example. Set
\[
\phi_\varepsilon(y):=\begin{cases}
  A\varepsilon,& |y|<\varepsilon,\\
  A|y|,& |y|\ge\varepsilon,
\end{cases}
\]
where $A\in\mathbb{R}$ and $\varepsilon>0$. Clearly, $\phi_\varepsilon(y)\to\phi_{A,A}(y)$ as $\varepsilon\downarrow0$ uniformly and $\phi_\varepsilon(y)=\phi_{A,A}(y)$ for $|y|\ge\varepsilon$.
 But, for $\phi=\phi_\varepsilon$, since
\[
|x|= \sqrt{1+A^2}y,\quad s_0=\sqrt{1+A^2}y-\varepsilon\left(\sqrt{1+A^2}-1\right)
\]
for $y>\varepsilon$, the condition \eqref{mustasymp} is not satisfied. Note that although $\phi_\varepsilon$ is not a smooth function, we can give a similar counterexample which is smooth by mollifying $\phi_\varepsilon$.

 In any case, the meaning of ``close'' should be clarified to claim non-existence of profile functions of self-similar solutions near $\phi_{A,B}$.
 Moreover, even if there exists no profile function near $\phi_{A,B}$, it does not imply non-existence of the profile function of self-similar solution $U(x,t)$ such that $\lim_{t\downarrow0}U(x,t)=\phi_{A,B}(x)$.
\end{rem}

\section{ Comparison with existence results of forward self-similar solutions} \label{SCE} 

Let us recall an existence result of a small forward self-similar solution due to \cite{KL}.
 It is stated explicitly only for the Willmore flow \cite[Theorem 3.12]{KL} but authors of \cite{KL} remarked in \cite[\S3.4]{KL}, the arguments for the surface diffusion flow are identical.
 Let us state their theorem for the surface diffusion flow.
 Their key norm for a function $u=u(x,t)$ is
\[
	\|u\|_{X_\infty} := \sup_{t>0} \left\| \nabla u(t) \right\|_{L^\infty(\mathbb{R}^n)}
	+ \sup_{x\in\mathbb{R}^n} \sup_{R>0}
	R^{\frac{2}{n+6}} \| \nabla^2 u \|_{L^{n+6}\left(B_R(x)\times\left( \frac{R^4}{2},R^4 \right)\right)},
\]
where $B_R(x)$ is a closed ball of radius $R$ centered at $x\in\mathbb{R}^n$ and $n=d-1$.
 They consider a graph-like solution for surface diffusion equation for general dimensions.
\begin{thm} \label{TKL}
\cite[\S3.4]{KL}. 
 There exist $\varepsilon>0$, $C>0$ such that if a Lipschitz function $u_0$ in $\mathbb{R}^n$ is self-similar (i.e., $u_0(x)=\Omega\left(x/|x|)\right|x|$ with some Lipschitz function $\Omega$ on $S^{n-1}$) and satisfies $\|\nabla u_0\|_{L^\infty}<\varepsilon$, then there exists a globally analytic and self-similar solution $u\in X_\infty$ of the surface diffusion flow with initial data $u_0$ satisfying the estimates $\|u\|_{X_\infty}\le C\|\nabla u_0\|_{L^\infty}$.
 (In particular, \eqref{ES1} holds for $\phi=u(\cdot,1)$ with $n=d-1$.)
 The solution is unique in $\|u\|_{X_\infty}\le C\varepsilon$.
 Moreover, there exists constants $R>0$ and $c>0$ such that
\[
	\sup_{x\in\mathbb{R}^n} \sup_{t>0} \left| (t^\frac14\nabla)^\alpha 
	(t\partial_t)^k \left| \nabla u (x,t) \right| \right|
	\le c \| \nabla u_0 \|_{L^\infty} 
	R^{|\alpha|+k} \left( |\alpha| + k \right)!
\]
for all $\alpha\in\mathbb{N}_0^n$ and $k\in\mathbb{N}_0$, where $\mathbb{N}_0=\{0,1,2,\ldots\}$.
\end{thm}
\begin{rem} \label{RSmall}
If we write $u_0(x)=\Omega\left(x/|x|\right)|x|$, then the smallness condition $\|\nabla u_0\|_{L^\infty}<\varepsilon$ holds provided that
\[
	\sup_{S^{n-1}} \left( |\Omega| + |\nabla_{S^{n-1}} \Omega | \right) < \varepsilon.
\]
In one-dimensional setting, if $u_0(x)=\phi_{A,B}(x)$, this condition is equivalent to $\max\left(|A|,|B|\right)<\varepsilon$.
\end{rem}
\begin{rem} \label{RSSI}
Since $u$ in Theorem \ref{TKL} is self-similar, i.e., $u(x,t)=t^{1/4}\phi(x/t^{1/4})$, the condition \eqref{ES1} is equivalent to saying that
\[
	\sup_{x\in S^{n-1}} \left| u(x,t) - u_0(x) \right| \to 0
	\quad\text{as}\quad t\downarrow 0.
\]
\end{rem}

From the statement of Theorem \ref{TKL}, it gives the impression that the derivative of a solution is uniquely determined but solution is determined up to some spatially constant function since norm of homogeneous Lipschitz function is used.
 In \cite[Remark 16]{RW}, P.\ Rybka and G.\ Wheeler conjectured that in one-dimensional setting there exist no trivial graph-like self-similar solutions close to $\phi_{A,B}$ and existence result in \cite{KL} would only yield the existence of special solution to the surface diffusion flow which is self-similar \emph{up to} some non-zero time-dependent spatially constant function.
 However, as remarked in Remark \ref{RCRW}, it is impossible to assert non-existence of a self-similar solution ``close'' to $\phi_{A,B}$ when $A\neq-B$.
 The question is whether a self-similar solution is obtained with no ambiguity of time-dependent spatially constant functions from a self-similar solution of the derivative equation.
 We shall show that the answer is yes and this disproves the conjecture of \cite{RW}. 

We point out that even if we only know the existence of a solution up to some spatially constant time-dependent function, we easily find a unique solution of the original equation with given initial data $u_0$.
 For simplicity, we restrict ourselves into one-dimensional setting.
 The graphical form of the surface diffusion equation \eqref{ESD} is of the form
\begin{equation}\label{EGSDI}
u_t=-\left(\frac{1}{(1+u_x^2)^{1/2}}\left(\frac{u_{xx}}{(1+u_x^2)^{3/2}}\right)_x\right)_x,\quad x\in\mathbb{R},\ t>0.
\end{equation}
Differentiating this equation in $x$ and setting $u_x=v$, we obtain
\begin{equation}\label{EGDSD}
	v_t = \left( H(v, v_x, v_{xx}) \right)_{xx}, \quad
	H(v, v_x, v_{xx}) = -\frac{1}{(1+v^2)^{1/2}} \left( \frac{v_x}{(1+v^2)^{3/2}} \right)_x.
\end{equation}
Since the right-hand side of \eqref{EGSDI} is $H(v,v_x,v_{xx})_x$, the solution $u$ of \eqref{EGSDI} is formally obtained by a simple integration
\[
	u(x,t) = u_0(x) + \int_0^t H(v,v_x,v_{xx})_x\,ds.
\]
The point is that the right-hand side of \eqref{EGSDI} depends only on $v$ and its derivatives, but it is independent of $u$ itself.

A more precise statement is as follows.
 Let $e^{-t\partial_x^4}$ denote the solution semigroup of bi-harmonic diffusion equation $\partial_tw=-\partial_x^4w$ with initial data $w_0$, i.e., $w(x,t)=(e^{-t\partial_x^4}w_0)(x)$.
 Using $e^{-t\partial_x^4}$, we write \eqref{EGSDI} and \eqref{EGDSD} in an integral form as in \cite{KL}.
\begin{thm} \label{TMKL}
Let $v$ be a solution of \eqref{EGDSD} with initial data $v_0=u_{0x}$ where $u_0$ is Lipschitz continuous in  the sense that it satisfies the integral equation
\begin{equation} \label{EEV}
	v = e^{-t\partial_x^4} v_0 +\int_0^t\partial_x^2e^{-(t-s)\partial_x^4}(\alpha[v]v_{xx}+F[v])(s)\,ds,
\end{equation}
where
\begin{equation} \label{EEVA}
	\alpha[v] := 1-\frac{1}{(1+v^2)^2} = \frac{2v^2+v^4}{(1+v^2)^2},\quad 
	F[v]:=\frac{3vv_x^2}{(1+v^2)^3}
\end{equation}
with the estimates
\begin{equation} \label{EDec}
	\left| \partial_x^\ell v(x,t) \right| \le C_\ell t^{-\ell/4}
	\quad\text{for}\quad \ell = 0, 1, 2,
\end{equation}
where $C_\ell>0$ is a constant independent of $t>0$.
 If we set
\begin{equation} \label{EDU}
	U(x,t) := e^{-t\partial_x^4} u_0 + \int_0^t\partial_x e^{-(t-s)\partial_x^4} \left( \alpha[v]v_{xx}+F[v] \right)(s)\,ds, 
\end{equation}
then $U$ is uniformly continuous up to $t=0$ and $U$ satisfies \eqref{EGSDI}.
 In other words, $u=U$ satisfies
\begin{equation} \label{ESDInt}
	u(x,t) = e^{-t\partial_x^4}u_0 + \int_0^t \partial_x e^{-(t-s)\partial_x^4}
	\left( \alpha[u_x]u_{xxx} + F[u_x] \right) (s)\,ds.
\end{equation}
If $U+h$ for $h=h(t)$ solves \eqref{ESDInt}, $h$ must be zero.
 The solution of \eqref{ESDInt} solves \eqref{EGSDI} with
\begin{equation} \label{AE}
	\sup_{x\in\mathbb{R}^n, t\in(0,1)}
	\left| U(x,t) - u_0(x) \right| \bigm/
	t^{1/4} < \infty.
\end{equation}
In particular, $U(x,t)\to u_0(x)$ as $t\downarrow0$ uniformly in $x\in\mathbb{R}^n$. 
\end{thm}
\begin{proof}
	By \eqref{EDec}, we obtain
\[
	|\alpha[v]v_{xx}|\le Ct^{-1/2},\quad 
	|F[v]| \le Ct^{-1/2} \quad\text{for}\quad t>0,
\]
where $C$ is a positive constant depending only on $C_\ell$'s.
 Using a regularizing estimate (see e.g.\ \cite{KL})
\begin{equation} \label{AER}
	\| \partial_x^\ell e^{-t\partial_x^4} f \|_{L^\infty(\mathbb{R}^n)}
	\le c_\ell t^{-\ell/4} \|f\|_{L^\infty(\mathbb{R}^n)}, \quad
	t > 0, \quad \ell = 0, 1, 2, \ldots,
\end{equation}
we conclude that
\begin{align*}
	\left\| \int_0^t \partial_x e^{-(t-s)\partial_x^4}
	\left( \alpha[v] v_{xx} + F[v] \right) (s)\, ds \right\|_{L^\infty}
	&\le \int_0^t  \frac{c_1}{(t-s)^{1/4}} \frac{2C}{s^{1/2}}\, ds \\
	&= 2c_1 Ct^{1/4} \int_0^1 (1-\tau)^{-1/4} \tau^{-1/2}\, d\tau < \infty.
\end{align*}
Thus, the function $U$ is well-defined.
 (Similar argument implies that the integrand of \eqref{EEV} is integrable as an $L^\infty$ valued function.
 Indeed,
\[
	\left\| \partial_x^2 e^{-(t-s)\partial_x^4}
	\left( \alpha[v] v_{xx} + F[v] \right) (s) \right\|_\infty
	\le \frac{c_2}{(t-s)^{1/2}} \frac{2C}{s^{1/2}},
\]
which is integrable on $(0,t)$.)
 Moreover, 
\begin{equation} \label{AED}
	\left\| U(x,t) - e^{-t\partial_x^4} u_0 \right\|_{L^\infty}
	\le C' t^{1/4}
\end{equation}
with $C'>0$ independent of $t$.
 Since
\begin{align*}
	e^{-t\partial_x^4} u_0 - u_0
	&= - \int_0^t \partial_s (e^{-s\partial_x^4} u_0)\, ds 
	= - \int_0^t \partial_x^4 (e^{-s\partial_x^4} u_0)\, ds \\
	& = - \int_0^t \partial_x^3 (e^{-s\partial_x^4} u_{0x})\, ds
\end{align*}
by $(e^{-s\partial_x^4} u_0)_x=(e^{-s\partial_x^4} u_0)$, we see, by \eqref{AER}, that
\[
	\left\| e^{-t\partial_x^4} u_0 - u_0 \right\|_{L^\infty}
	\le c_3 \int_0^t \frac{ds}{s^{3/4}} \|u_{0x}\|_{L^\infty}
	= 4c_3 t^{1/4} \|u_{0x}\|_{L^\infty}.
\]
Combining this estimate with \eqref{AED}, we conclude \eqref{AE}.
 Furthermore, differentiating both sides of \eqref{EDU}, we see that
\begin{equation}\label{Urepr}
U_x(x,t)=(e^{-t\partial_x^4} u_0)_x+\int_0^t \partial_{x}^2e^{-(t-s)\partial_x^4}(\alpha[v]v_{xx}+F[v])(s)\,ds.
\end{equation}
Since $(e^{-t\partial_x^4}u_0)_x=e^{-t\partial_x^4}v_0$, this implies that $U_x=v$ by \eqref{EEV}.
 Plugging $v=U_x$ in \eqref{EDU}, we conclude that $u=U$ satisfies \eqref{ESDInt}.
 Since the value of the right-hand side of \eqref{ESDInt} is the same for $U+h$ and $U$, $h$ must be zero by \eqref{ESDInt}.

It remains to show that the solution of the integral equation actually solves \eqref{ESDInt}.
 This is standard, so we only give a formal argument.
 We notice that
\[
	w = e^{-t\partial_x^4} w_0
	+ \int_0^t e^{-(t-s)\partial_x^4} f \,ds
\]
solves
\[
	\partial_t w = -\partial_x^4 w + f, \quad
	\left. w \right|_{t=0} = w_0
\]
by Duhamel's formula.
Since $\partial_x e^{-t\partial_x^4}=e^{-t\partial_x^4}\partial_x$, this equivalence implies that
\[
	\partial_t u = -\partial_x^4 u 
	+ \partial_x \left( \alpha[u]u_{xx} + F[u_x] \right), \quad
	\left. u \right|_{t=0} = u_0.
\]
It is not difficult to see that this equation is the same as \eqref{EGSDI} by explicit manipulation.
 The proof is now complete.
\end{proof}
\begin{rem}[Uniqueness] \label{RUniq}
The uniqueness of a solution of \eqref{EEV} under the estimates \eqref{EDec} is guaranteed provided that $C_\ell$ in \eqref{EDec} is small, which is natural if we assume that $\|v_0\|_{L^\infty}$ is small as discussed in \cite{KL}.
 If the uniqueness for \eqref{EEV} is guaranteed, our argument implies that a solution of \eqref{ESDInt} is also unique.
 Indeed, let $u^1$ and $u^2$ be two solutions of \eqref{ESDInt} satisfying \eqref{EDec} for $v^i=u_x^i$ ($i=1,2$) with small $C_\ell$ (so that the uniqueness of \eqref{EEV} holds).
 Since $v^i=u_x^i$ ($i=1,2$) solves \eqref{EEV}, the uniqueness of \eqref{EEV} yields $u_x^1=u_x^2$.
 Thus the right-hand side of \eqref{ESDInt} is the same for $u^1$ and $u^2$, \eqref{ESDInt} yields that $u^1=u^2$.
\end{rem}

Our assertion for the equation \eqref{ESDInt} has no ambiguity for time-dependent spatially constant function.
 Indeed, if $u_0=\phi_{A,B}$, because of homogeneity, $\sigma^{-1/4}U(\sigma^{1/4}x,\sigma t)$ for $\sigma>0$ is a solution to \eqref{ESDInt}.
 Thus, by Theorem \ref{TMKL}, we see that
\[
	\sigma^{-1/4} U(\sigma^{1/4}x, \sigma t)	= U(x,t), \quad
	x \in \mathbb{R}, \quad
	t > 0;
\]
in other words, $U$ is a self-similar solution to \eqref{ESDInt}.

We give a direct way to show that $U$ is a self-similar solution without using the uniqueness of the solution of \eqref{EGSDI}.
 We first notice that $v$ is self-similar in the sense that
\begin{equation} \label{ESSV}
  v^\sigma(x,t):=v(\sigma^{1/4}x,\sigma t)=v(x,t),\quad x\in\mathbb{R},\ t>0,\ \sigma>0.
\end{equation}
We prove that $U$ is a self-similar solution to the equation \eqref{EGSDI}. By \eqref{EDU} and \eqref{EEV}, $U$ satisfies
\begin{align*}
  U_t+\partial_x^4U&=\left(\left(1-\frac{1}{(1+v^2)^2}\right)v_{xx}+\frac{3vv_x^2}{(1+v^2)^3}\right)_x=\partial_x^3v-\left(\frac{1}{(1+v^2)^{1/2}}\left(\frac{v_x}{(1+v^2)^{3/2}}\right)_x\right)_x\\
  &=\partial_x^4U-\left(\frac{1}{(1+U_x^2)^{1/2}}\left(\frac{U_{xx}}{(1+U_x^2)^{3/2}}\right)_x\right)_x,\quad x\in\mathbb{R},\ t>0,
\end{align*}
and thus $U$ is a solution to \eqref{EGSDI}.
 Since
\begin{align*}
  &[e^{-t\partial_x^4}f](\sigma^{1/4}x)
  = \left[e^{-\sigma^{-1}t\partial_x^4}
  f(\sigma^{1/4}\cdot)\right](x), \\ 
  &[\partial_x e^{-t\partial_x^4}f](\sigma^{1/4}x)
  = \sigma^{-1/4} \left[\partial_x e^{-\sigma^{-1}t\partial_x^4} f (\sigma^{1/4}\cdot) \right] (x)
\end{align*}
for any function $f$, we see that
\begin{align*}
  &\sigma^{-1/4}U(\sigma^{1/4}x,\sigma t)\\
  &=\sigma^{-1/4}[e^{-\sigma t\partial_x^4}\phi_{A,B}](\sigma^{1/4}x)+\sigma^{-1/4}\int_0^{\sigma t}\left[\partial_xe^{-(\sigma t-s)\partial_x^4} \left(\alpha[v]v_{xx}+F[v]\right)(s)\right](\sigma^{1/4}x)\,ds\\
  &=\sigma^{-1/4}e^{-t\partial_x^4}\phi_{A,B}(\sigma^{1/4}\cdot) +\sigma^{-1/2} \int_0^{\sigma t} \left[\partial_x e^{-(t-\sigma^{-1}s)\partial_x^4} \left(\alpha[v]v_{xx} + F[v] \right)
  (\sigma^{1/4}\cdot,s) \right](x)\,ds\\
  &= I + I\!I.
\end{align*}
Since $\phi_{A,B}(\sigma^{1/4}x)=\sigma^{1/4}\phi_{A,B}(x)$, we see that
\[
	I = \sigma^{-1/4} e^{-t\partial_x^4} \sigma^{1/4} \phi_{A,B} 
	= e^{-t\partial_x^4} \phi_{A,B}.
\]
Since
\begin{align*}
	& \partial_x \left(f (\sigma^{1/4}x) \right)
	= \sigma^{1/4} f_x (\sigma^{1/4}x), \\
	& \partial_x^2 \left(f (\sigma^{1/4}x) \right)
	= \sigma^{1/2} f_{xx} (\sigma^{1/4}x),
\end{align*}
we see that
\begin{align*}
	& \alpha \left[ f (\sigma^{1/4} \cdot) \right](x)
	= \alpha [f] (\sigma^{1/4} x) \\
	& F \left[f (\sigma^{1/4} \cdot) \right](x)
	= \sigma^{1/2} F[f] (\sigma^{1/4}x).
\end{align*}
We thus observe that
\begin{align*}
  I\!I &= \sigma^{-1/2} \int_0^{\sigma t} \partial_x e^{-(t-\sigma^{-1}s)\partial_x^4} \sigma^{-1/2}
  \left( \alpha \left[v(\sigma^{1/4}\cdot,s)\right] v_{xx} (\sigma^{1/4}\cdot,s)
  + F \left[ v(\sigma^{1/4}\cdot,s)\right] \right) ds \\
  & = \sigma^{-1} \int_0^t \partial_x e^{-(t-\tau)\partial_x^4}
  \left( \alpha[v^\sigma](v^\sigma)_{xx} + F [v^\sigma] \right) \sigma\, d\tau;
\end{align*}
here we changed the variable of integration by $s=\sigma\tau$.
 Since $v^\sigma=v$ by \eqref{ESSV}, we conclude that
\begin{align*}
  \sigma^{-1/4}U(\sigma^{1/4}x, \sigma t) &= I + I\!I \\
  &= e^{-t\partial_x^4} \phi_{A,B} + \int_0^t \partial_x e^{-(t-\tau)\partial_x^4}
  \left(\alpha[v]v_{xx} + F[v]\right)d\tau \\
  &=U(x,t).
\end{align*}
Hence $U$ is a self-similar solution to \eqref{EGSDI}.

In summary, we obtain
\begin{thm} \label{TS}
Let $v$ be a self-similar solution to \eqref{EGDSD} in the sense of \eqref{ESSV} satisfying \eqref{EDec} with $v(x,0)=-B$ for $x<0$ and $v(x,0)=A$ for $x>0$, where $A,B\in\mathbb{R}$.
 Then $U$ given by \eqref{EDU} is a self-similar solution to \eqref{EGSDI} (which is uniformly continuous up to $t=0$) with initial data $U(x,0)=\phi_{A,B}(x)$.
 In particular, there exists a forward self-similar solution, which may not be linear provided that $A$ and $B$ are small. 
\end{thm}


\section*{Acknowledgements}
The authors are grateful for Professor Piotr Rybka for valuable discussion and comments on this topic.
 The work of the first author was partly supported by Japan Society for the Promotion of Science (JSPS) through grants KAKENHI Grant Numbers 24K00531, 24H00183 
and by Arithmer Inc., Daikin Industries, Ltd.\ and Ebara Corporation through collaborative grants. 
 The work of the second author was supported by Grant-in-Aid for JSPS Fellows DC1, Grant Number 23KJ0645.


\end{document}